\documentclass{article}
\usepackage{amsfonts,amssymb,amsmath,graphicx,xcolor,colortbl, verbatim}
\usepackage{amsthm}
\usepackage{fullpage}
\def\revisedtext#1{ #1}
\newcommand{\R}{\mathbb{R}}

\usepackage{mdframed}
\usepackage{thmtools}

\definecolor{shadecolor}{gray}{0.90}
\declaretheoremstyle[
headfont=\normalfont\bfseries,
notefont=\mdseries, notebraces={(}{)},
bodyfont=\normalfont,
postheadspace=0.5em,
spaceabove=1pt,
mdframed={
  skipabove=8pt,
  skipbelow=8pt,
  hidealllines=true,
  backgroundcolor={shadecolor},
  innerleftmargin=4pt,
  innerrightmargin=4pt}
]{shaded}

\declaretheorem[style=shaded,within=section]{definition}

\declaretheorem[style=shaded,sibling=definition]{proposition}

\declaretheorem[style=shaded,sibling=definition]{corollary}

\begin{document}



\title{A new framework for the computation of Hessians\footnote{This work was published in 2012~\cite{Gower:2012}}}

\author{R. M. Gower$^{\rm a}$ $^{\ast}$\footnote{$^\ast$Corresponding author. Email: gowerrobert@gmail.com. Partially supported by
CNPq and FAPESP (Grant 2009/04785-7)} \\
 State University of Campinas
  \and M. P. Mello$^{\rm a}$ $^{\dagger}$\footnote{$\dagger$Partially supported by CNPq-PRONEX Optimization and FAPESP (Grant
2006/53768-0).}\\ State University of Campinas}

\date{15 August 2011}
\maketitle

\begin{abstract}
We investigate the computation of Hessian matrices via Automatic Differentiation, using a graph model and an
algebraic model. The graph model reveals the inherent symmetries involved in calculating the Hessian. The algebraic model, based
on Griewank and Walther's state transformations~\cite{Griewank:2008}, synthesizes the calculation of the Hessian as a formula.
These dual points of view, graphical and algebraic, lead to a new framework for Hessian computation. This is illustrated by
developing \texttt{edge\_pushing}, a new truly reverse Hessian computation algorithm that fully exploits the Hessian's symmetry.
Computational experiments compare the performance of \texttt{edge\_pushing} on sixteen functions from the CUTE collection \cite{Cute} against two algorithms available as drivers of the software ADOL-C \cite{ADOL-C,AndreaSparsity,GebreHessianAuto}, and the results are very promising.\bigskip


\end{abstract}

\section{Introduction}

Within the context of nonlinear optimization, algorithms that use variants of Newton's method
 must repeatedly calculate or obtain approximations of the Hessian matrix or Hessian-vector products. 
Interior-point methods, ubiquitous in nonlinear solvers~\cite{Forsgreninteriormethods}, fall in this category.
While the nonlinear optimization package LOQO~\cite{VanderbeiInteriorPoint} requires  
that the user supply the Hessian, IPOPT~\cite{BieglerLineSearch} and KNITRO~\cite{knitro} are more flexible, but
also use Hessian information of some kind or other. Experience indicates that optimization algorithms that employ first order
derivatives perform fewer iterations given exact gradients, as opposed to numerically approximated ones. Although there is not an
equivalent consensus concerning second order derivatives, it is natural to suspect the same would hold true for algorithms that
use Hessians. Thus the need to efficiently calculate exact (up to machine precision) Hessian matrices is driven by the rising
popularity of optimization methods that take advantage of second-order information.

Automatic Differentiation (AD) has had a lot of success in calculating gradients and Hessian-vector products with reverse AD
procedures~\cite{Christianson:1992}\footnote{Reverse in the sense that  the order of evaluation is opposite to the
order employed in calculating a function value.} that have the same time complexity as that of evaluating the underlying
function.

Attempts to efficiently calculate the entire Hessian matrix date back to the work of Jackson and McCormick
\cite{McCormickJacksonFactorable}, based on Jackson's dissertation.
Their work was followed by increasingly intense research in this area, no doubt helped along by the advances in hardware and
software. Since the beginning, exploring sparsity and symmetry were at the forefront of efficiency related issues. Nowadays we can
discern a variety of strategies in the literature, regarding how to properly take these into account. The authors
of~\cite{McCormickJacksonFactorable} explore sparsity and symmetry by storing and operating on the Hessian in an outer product
format, the so called dyadic form.
A natural strategy, when employing a forward Hessian mode, is to store the Hessian matrices
involved in data structures that accommodate their symmetry~\cite{Abate:1997}.
When dealing with very sparse matrices, one may obtain the sparsity pattern,
and then individually calculate selected nonzero
elements using methods such as univariate Taylor expansion~\cite{Abate:1997,Bischof:1993}.
Truly effective methods currently in use, with a substantial number of reports including numerical tests, take advantage of
sparsity and symmetry by combining graph coloring with Hessian-vector AD routines~\cite{GebreHessianAuto,AndreaSparsity}.

The paper is organized as follows. Section~\ref{sec:preliminaries} presents concepts and notation regarding function and
gradient evaluation in AD. The graph model for Hessian computation is developed in
Section~\ref{sec:HessianGraphModel} and the
algebraic formula for the Hessian is obtained in the next section. The new algorithm, \texttt{edge\_pushing}, or
\texttt{e\_p} for short, is described in Section~\ref{sec:ep}. The computational experiments are reported in
Section~\ref{sec:computationalexperiments} and we close with conclusions and comments on future work.

\section{Preliminaries: function and gradient computation}\label{sec:preliminaries}
In order to simplify the discussion, we consider functions $f\colon\R^n \to \R$ that are twice continuously differentiable. It is
more convenient and the results obtained can be generalized in a straightforward manner to smaller domains and functions that are
twice continuously differentiable by parts. There are of course multiple possibilities for expressing a function, but even if one
chooses a specific way to write down a function, or a specific way of programming a function $f$, one still may come up with
several distinct translations of $f$ into a finite sequential list of functions. We assume in the following that such a list has
already been produced, namely there exists a sequence $(\phi_{1-n},\ldots, \phi_0, \phi_1,\ldots, \phi_\ell)$, such that the first
$n$ functions are the coordinate variables, each \emph{intermediate function} $\phi_i$, for $i=1,\ldots, \ell$, is a function of
previous functions in the sequence, and, if we sweep this sequence in a forward fashion, starting with some fixed vector
$x=(\phi_{1-n},\ldots,\phi_0)$, the value obtained for $\phi_\ell$ coincides with the value of $f(x)$. Jackson and McCormick
\cite{McCormickJacksonFactorable} dealt with a very  similar concept, which they called a \emph{factorable function}, but in that
case the intermediate functions were either sums or products of precisely two previous functions, or generic functions of a single
previous function, that is, unary functions. Although the framework for calculating the Hessian developed here is valid for
intermediate functions with any number of input variables, when evaluating complexity bounds, we assume that the functions
$\phi_i$, for $i=1,\ldots, \ell$, are either unary or binary.

It is very convenient to model the sequential list $(\phi_{1-n}, \ldots, \phi_\ell)$ and the interdependence amongst its
components as an acyclic digraph $G=(N,A)$, called \emph{computational graph}. Loosely speaking, the computational graph
associated with the list has nodes $\{1-n,\ldots, \ell\}$ and edges $\{(j,i)\mid \phi_i \mbox{ depends on the}$ $\mbox{value of  }
\phi_j\}$. The interdependence relations are thus translated into predecessor relations between nodes, and are denoted by the
symbol $\prec$. Thus the arc $(j,i)$ embodies the precedence relation $j\prec i$. Notice that, by construction, $j\prec i$ implies
$j<i$. Furthermore, if we denote by $v_i$ the output value of $\phi_i$ for a given input, then we may shorten, for instance, the
expression $v_i=\phi_i(v_j,v_k)$ to $v_i=\phi_i(v_j)_{j\prec i}$.

\begin{figure}
\centering\includegraphics[scale=.82
]{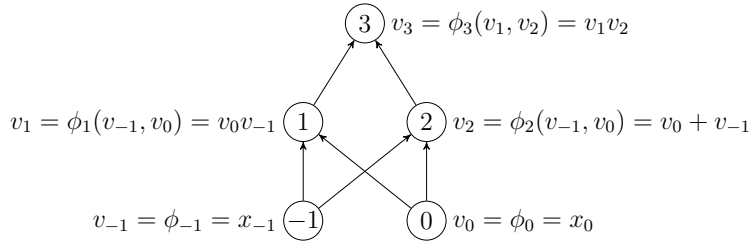}
\caption{Computational graph of the function $f(x)=(x_{-1} x_0)(x_{-1} + x_0)$.}\label{fig:computationalgraph}
\end{figure}

Due to the choice of the numbering scheme for the $\phi$'s, commonly adopted in the literature, we found it convenient to apply,
throughout this article, a shift of $-n$ to the indices of all matrices and vectors. We already have $x\in \R^n$, which, according
to this convention, has components $x_{1-n}$, $x_{2-n}$, \ldots, $x_0$. Similarly, the rows/columns of the Hessian $f''$ are
numbered $1-n$ through~$0$.  Other vectors and matrices will be gradually introduced, as the need arises for expressing and
deducing mathematical properties enjoyed by the data.  Figure~\ref{fig:computationalgraph} shows the computational graph of
function $f(x)=(x_{-1} x_0)(x_{-1} + x_0)$ that corresponds to the sequence $\revisedtext{(v_{-1},\ldots,v_3)}=
(x_{-1},x_0,x_{-1}x_0,x_{-1}+x_0,(x_{-1} x_0)(x_{-1} + x_0))$.

Griewank and Walther's \cite{Griewank:2008} representation of $f$ as a composition of state transformations
\begin{equation}\label{eq:StateTransform}
 f(x) = e_{\ell}^T (\Phi_{\ell} \circ \Phi_{\ell-1} \circ \cdots \circ \Phi_{1}) (P^{T} x), 
\end{equation}
\revisedtext{where $e_{\ell}$ is the $(\ell + n)$th canonical vector,} the $n\times (n+\ell)$ matrix $P$ is zero except for the
leftmost $n$-dimensional block which contains an identity matrix, and
\begin{eqnarray}
\Phi_i &\colon& \R^{n+\ell}\to\R^{n+\ell}\nonumber\\
&& \parbox{25pt}{\centerline{$y$}} \mapsto  (y_{1-n},\ldots,y_{i-1},\phi_i(y_j)_{j\prec i},y_{i+1},\ldots,
y_\ell)^T\label{eq:statetrasnfPhi}\end{eqnarray} 
leads to a synthetic formula for the gradient of $f$, using the chain rule recursively:
\begin{equation}
\label{eq:gradientstatetransform}
(\nabla f(x))^T = e_{\ell}^T \Phi_{\ell}'\Phi_{\ell-1}' \cdots \Phi_{1}'P^{T}.
\end{equation}
For simplicity's sake, the argument of each function is omitted in (\ref{eq:gradientstatetransform}), but it should
be noted that $\Phi_k'$ is evaluated at $(\Phi_{k-1}\circ\Phi_{k-2}\circ\cdots\circ\Phi_1)(P^T x)$, for $k=1,\ldots, \ell$.

The advantage of vector/matrix notation is that formulas expressed in terms of vector/matrix operations usually lend themselves
to straightforward algorithmic implementations. Nevertheless, when analyzing complexity issues and actual implementation, one has
to translate block operations with vectors or matrices into componentwise operations on individual variables.

In this case, for instance, one can immediately devise two ways of obtaining $\nabla f$ based
on~(\ref{eq:gradientstatetransform}): calculating the product  of the matrices in a right-to-left fashion, or left-to-right. The
latter approach constituted a breakthrough in gradient computation, since the time complexity of its implementation was basically
the same as that of the function evaluation, a major improvement over the former approach. In the left-to-right way, the indices
are traveled in decreasing order, so this method of calculating the gradient is called the \emph{reverse gradient computation}.
Notice that, in graph terms, this corresponds to a backward sweep of the computational graph.

Of course, one needs the values of $\revisedtext{v_j}$, for $j\prec i$, in order to calculate $\revisedtext{v}_i$. Thus in
order to do perform a
backward sweep, it must be preceded by a forward sweep, in which all the values $\revisedtext{v_i}$, \revisedtext{for $i = 1,
\ldots, \ell$,} have been calculated. We shall call
the data structure that contains all information concerning the function evaluation produced during the forward
sweep a \emph{tape} $\mathcal{T}$. Thus the tape contains the relevant  recordings of a forward sweep along with the computational
graph of $f$.

Algorithms~1 and~2 contain the implementation of the reverse gradient computation in block and
componentwise forms, respectively.

\begin{center}
\includegraphics[scale=.83
]{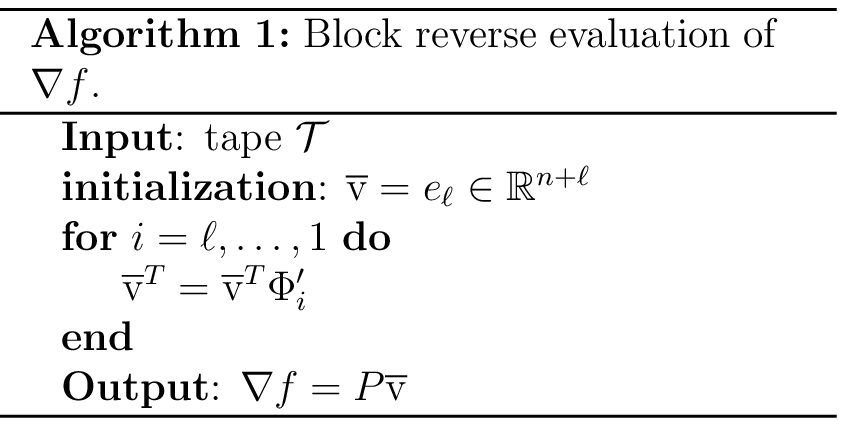}

\includegraphics[scale=.83
]{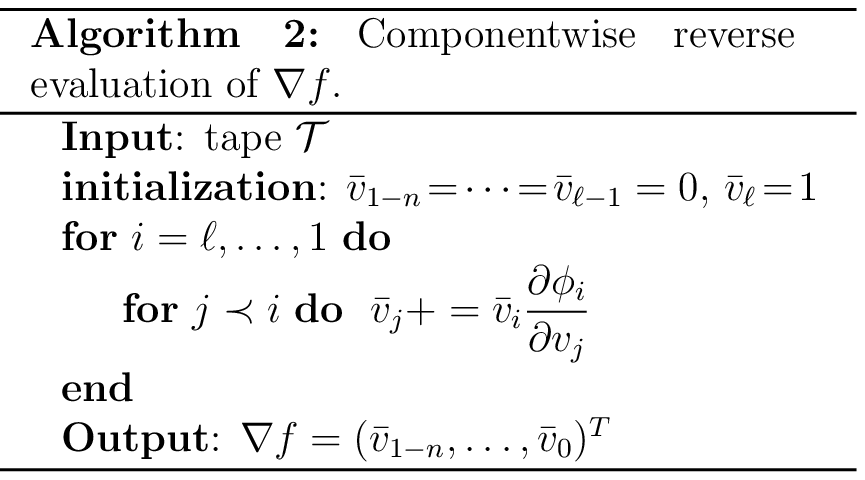}
\end{center}

\medskip
In Algorithm~1 the necessary partial products are stored in $\overline{\mathrm{v}}$, and, right
before node $i$ is swept, the vector $\overline{\mathrm{v}}$ satisfies
\begin{equation}\label{eq:adjointvectornoindex}
\overline{\mathrm{v}}^T = e^T_\ell \prod_{j=1}^{\ell-i} \Phi'_{\ell-j+1}.
\end{equation}

The streamlined componentwise form of Algorithm~1 follows from the very simple block structure of
the Jacobian $\Phi_i'$:
\begin{equation}\label{eq:PhiJacobian}
\Phi_i' = \left[ \begin{array}{@{\hspace{7pt}}c@{\hspace{8pt}}|@{\hspace{5pt}}c@{\hspace{5pt}}|@{\hspace{11pt}}c@{\hspace{7pt}}} 
\rule{0pt}{23pt}\rule[-13pt]{0pt}{10pt}I & 0 & 0\makebox(0,10)[lc]{\footnotesize$\begin{array}{@{\hspace{.5cm}}c}1-n\\[-5pt]
\vdots \\[-3pt] i-1\end{array}$}\\ 
\hline
(c^i)^T\rule{0pt}{12pt} & 0 & 0\makebox(0,10)[lc]{\footnotesize\hspace{.5cm}row $i$,} \\ \hline
\rule{0pt}{23pt}\rule[-10pt]{0pt}{10pt}0 & 0 & I\makebox(0,10)[lc]{\footnotesize$\begin{array}{@{\hspace{.5cm}}c}i+1\\[-5pt]
\vdots \\[-3pt] \ell\end{array}$}\end{array}\right]
\end{equation}
where
\begin{equation}\label{eq:definicaocT}
(c^i)_{ j} =\frac{\partial\phi_i}{\partial v_j},\quad\mbox{for }j=1-n,\ldots,i-1.
\end{equation}
Thus $(c^i)^T$ is basically the transposed gradient of $\phi_i$ padded with the convenient number of zeros at the appropriate
places.
In particular, it has at most as many nonzeros as the number of predecessors of node $i$, and the post-multiplication
$\overline{\mathrm{v}}^T\Phi'_i$ affects the components of $\overline{\mathrm{v}}$ associated with the predecessors of node $i$
and zeroes component $i$. In other words, denoting component $i$ of $\overline{\mathrm{v}}$ by $\bar{v}_i$, the block assignment
in Algorithm~1 is equivalent to
\[\bar{v}_j \leftarrow\left\{\begin{array}{ll}\bar{v}_j +\bar{v}_i\dfrac{\partial\phi_i}{\partial v_j},\quad
&\mbox{if }j\prec i,\\
0, &\mbox{if }j=i,\\[5pt]
\bar{v}_j,& \mbox{otherwise.} \end{array}\right.\]
Now this assignment is done as the node $i$ is swept, and, therefore, in subsequent iterations component $i$ of
$\overline{\mathrm{v}}$ will not be accessed, since the loop visits nodes in decreasing index order. Hence setting component $i$
to zero has no effect on the following iterations. Eliminating this superfluous reduction, we arrive at Algorithm
~2, the componentwise (slightly altered) version of
Algorithm~1.

In order to give a graph interpretation of Algorithm~2, let $\revisedtext{c^{i}_{j}} = \partial \phi_i/\partial
v_j$ be the weight of arc $(j,i)$, and define the weight of a directed path from node $j$ to node $k$ as the product of the
weights of the arcs in the path. Then, one can easily check by induction that, right before node $i$ is swept, the \emph{adjoint}
$\bar{v}_i$ contains the sum of the weights of all the paths from node $i$ to node $\ell$:
\begin{equation}\label{eq:barvipath} \bar{v}_i = \sum_{p\mid\mbox{\scriptsize path from $i$ to $\ell$}} \mbox{weight of path }
p,\end{equation}
and $\partial f/\partial v_i = \bar{v}_i$. This formula has been known for quite some time~\cite{Bauer_1974}.

As node $i$ is swept, the value of $\bar{v}_i$ is properly distributed amongst its predecessors, in the sense that, accumulated
in $\bar{v}_j$, for each predecessor $j$, is the contribution of all paths from $j$ to $\ell$ that contain node $i$, weighted by
$\revisedtext{c^{i}_{j}}$. Hence, at the end of Algorithm~2, the adjoint $\bar{v}_{i-n}$, for $i=1,\ldots, n$,
contains the sum of the weights of all paths from $i-n$ to $\ell$. This is in perfect accordance with the explanation for the
computation of partial derivatives given in some Calculus textbooks, see, for instance, \cite[p. 940]{Stewart:2007}.

Of course different ways of calculating the product of the Jacobians of the state transformations in
(\ref{eq:gradientstatetransform}) may give rise to different algorithms. In the following, using the same ingredients, we obtain a
closed formula for the Hessian, that can be used to justify known algorithms as well as suggest a new algorithm for Hessian
computation. Before that, however, we develop a graph understanding of the Hessian computation.

\section{Hessian graph model}\label{sec:HessianGraphModel}

Creating a graph model for the Hessian is also very useful, as it provides insight and intuition regarding the workings of Hessian
algorithms. Not only can the graph model suggest algorithms, it can also be very enlightening to interpret the operations
performed by an algorithm as operations on variables associated with the nodes and arcs of a computational graph.

Since second order derivatives are simply first order derivatives of the gradient, a natural approach to their calculation would
be to build a computational graph for the gradient and apply a variant of Algorithm~2 on this new graph to obtain
the second order partial derivatives. We do this to better understand the problem, but later on we will see that it is
not really necessary to build the full-fledged gradient computational graph, but one can instead work with the
original graph plus some new edges.

Of course the gradient may be represented by distinct computational graphs, or equivalently, sequential lists of functions, but
the natural one to consider is the one associated with the computation performed by
Algorithm~2. Assuming this choice, the gradient $\nabla f =
(\bar{v}_{1-n},\ldots,\bar{v}_0)^T$ is a composite function of $(\bar{v}_1,\ldots,\bar{v}_\ell)$, as well as
$(v_{1-n},\ldots,v_\ell)$, which implies that the gradient (computational) graph $G^g=(N^g,A^g)$ must contain $G$. The graph $G^g$
is basically built upon $G$ by adding nodes associated with $\bar{v}_i$, for $i=1-n,\ldots, \ell$, and edges 
representing the functional dependencies between these nodes.

Thus the node set $N^g$ contains $2(n+\ell)$ nodes $\{1-n,\ldots,\ell, \overline{1-n},\ldots, \overline{\ell}\}$, the first half
associated with the original variables and the second half with the adjoint variables. The arc set is $A^g=A_1\cup A_2\cup A_3$,
where $A_1$ contains arcs with both endpoints in ``original'' nodes; $A_2$, arcs with both endpoints in ``adjoint'' nodes and
$A_3$, arcs with endpoints of mixed nature. Since running Algorithm~2 does not introduce new
dependencies amongst the original $v$'s, we have that $A_1=A$. 

The new dependent variables created by running Algorithm~2 satisfy
\begin{equation}\label{eq:barvi}
\bar{v}_i = \sum_{k \mid i\prec k} \bar{v}_k\frac{\partial v_k}{\partial v_i}\end{equation}
at the end of the algorithm.
Expression (\ref{eq:barvi}) indicates that $\bar{v}_i$ depends on $\bar{v}_k$ for every $k$ that is a successor of
$i$. Thus every arc $(i,k)\in A_1$ gives rise to arc $(\bar{k},\overline{\textit{\i}})\in A_2$. Therefore arcs in $A_2$ are copies
of the arcs in $A$ with the orientation reversed. The graph $G^g$
thus contains $G$ and a kind of a  mirror copy of $G$. Furthermore, if a partial derivative $\partial v_k/\partial v_i$ in the
sum in (\ref{eq:barvi}) is not constant, but is a function of some $v_j$, this gives rise to the precedence relation $j\prec
\overline{\textit{\i}}$. Notice that this may only happen if $j\prec k$, that is, the arcs $(j,\overline{\textit{\i}})\in A_3$
imply that $j$ and $i$ share a common successor~$k$. This implies, in particular, that there are no edges incident to
$\overline{\ell}$.

Apparently, the computational graph of the gradient was first described in \cite{Dix91}, but can be found in a number
of places, e.g., \cite[p. 237]{Griewank:2008}.

Figure~\ref{fig:gradientcomputationalgraph} shows the computational graph of the gradient of the function
$f$ given in Figure~\ref {fig:computationalgraph}. Notice that on the left we have the computational graph of $f$ and, on the
right, a mirror copy thereof. Arcs in $A_3$ are the ones drawn dashed in the picture.

\begin{figure}
\begin{center}
\includegraphics[scale=.82
]{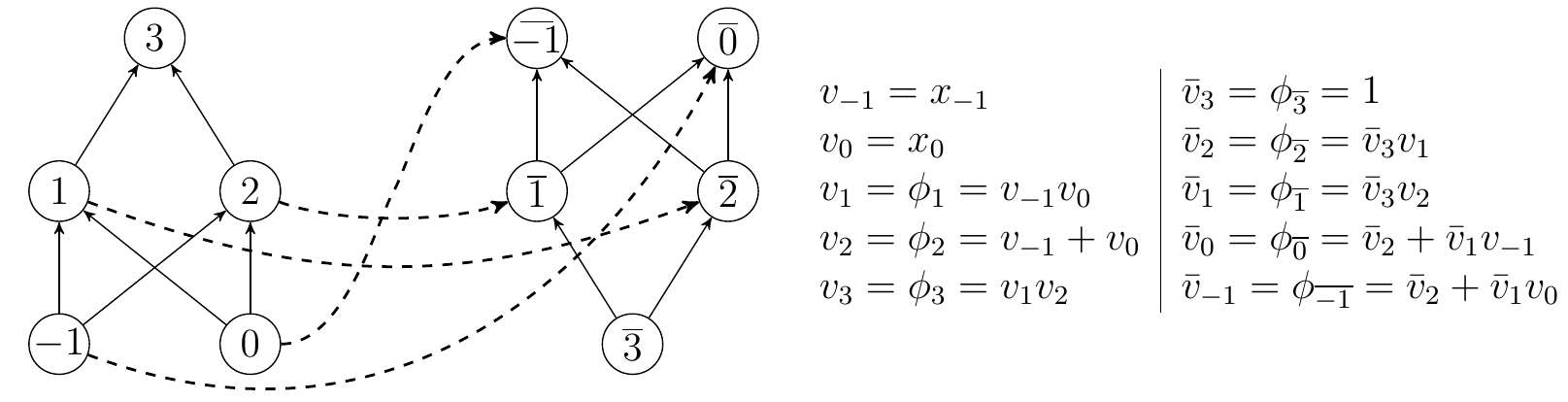}
\caption{Gradient computational graph $G^g$ of the function $f(x)=(x_{-1} x_0)(x_{-1} + x_0)$, represented
by the computational graph in Figure~\protect\ref{fig:computationalgraph}.}\label{fig:gradientcomputationalgraph}
 \end{center}
\end{figure}

Mimicking (\ref{eq:barvipath}), we conclude that
\begin{equation}\label{eq:derij}
\frac{\partial^2 f}{\partial x_i\partial x_j}  = \sum_{p\mid \mbox{\scriptsize path from $i$ to
$\overline{\textit{\j}}$}}\mbox{weight of path $p$}. \end{equation}

The weights of arcs $(i,j)\in A_1$ are already know. Equation (\ref{eq:barvi}) implies that the weight of arc
$(\overline{\textit{\i}}, \overline{\textit{\j}})\in A_2$ is
\begin{equation}\label{eq:symmetricweights}
c^{\kern1pt \overline{\textit{\j}}}_{\rule{0pt}{8pt}\overline{\textit{\i}}} =\frac{\partial v_i}{\partial v_j} =
c^{i}_{j},\end{equation}
that is, arc $(i,j)$ has the same weight as its mirror image.
 
The weight of arc $(j,\textit{\i})$ is also obtained from (\ref{eq:barvi})
\begin{eqnarray}
c^{\overline{\textit{\i}}}_{j} &=& \sum_{k\mid i\prec k} \bar{v}_k\frac{\partial^2 v_k}{\partial v_j\partial
v_i}\nonumber\\
& = &  \sum_{k\mid i\prec k~\mbox{\scriptsize and}~j\prec k} \bar{v}_k\frac{\partial^2 v_k}{\partial v_j\partial
v_i},\label{eq:streamliningw}
\end{eqnarray}
since the partial derivative $\partial^2 v_k/\partial v_j\partial v_i$ is identically zero if $k$ is not a successor of~$j$. In
particular, (\ref{eq:streamliningw}) and the fact that $f$ is twice continuously differentiable imply that
\begin{equation}\label{eq:pesosnaolineares}
c^{\overline{\textit{\i}}}_{j} =c^{\overline{\textit{\j}}}_{\rule{0pt}{8pt}i} , \mbox{ for }
j\neq i .
\end{equation}
Notice that arcs in $A_3$ are the only ones with second-order derivatives as weights. In a sense, they carry the nonlinearity of
$f$, which suggests the denomination \emph{nonlinear arcs}. 

Regarding the paths in $G^g$ from $i$ to $\overline{\textit{\j}}$, for fixed $i,j\in \{1-n,\ldots, 0\}$, each of them contains a
unique nonlinear arc, since none of the original nodes is a successor of an adjoint node. Therefore, the sum in (\ref{eq:derij})
may be partitioned according to the nonlinear arc utilized by the paths as follows:
\begin{equation}\label{eq:2derpartioned}
\frac{\partial^2 f}{\partial x_i\partial x_j}= \kern-9pt\sum_{(r,\overline{s})\in A_3}\!\! \left[\!
\left(\sum_{p\mid \mbox{\scriptsize path from $i$ to $r$}}\kern-20pt\mbox{weight of path $p$}\right) 
c^{\overline{s}}_{r}
\left(\sum_{q\mid \mbox{\scriptsize path from $\overline{s}$ to $\overline{\textit{\j}}$}}\kern-19pt\mbox{weight of path $q$}\right)\!\right]\!, \end{equation}
which reduces to 
\begin{equation}\label{eq:2derpartionedfolded}
\frac{\partial^2 f}{\partial x_i\partial x_j} = \kern-9pt\sum_{(r,\overline{s})\in A_3}\!\! \left[\!
\left(\sum_{p\mid \mbox{\scriptsize path from $i$ to $r$}}\kern-20pt\mbox{weight of path $p$}\right) 
c^{\overline{s}}_{r}
\left(\sum_{q\mid \mbox{\scriptsize path from $j$ to $s$}}\kern-19pt\mbox{weight of path $q$}\right)\!\right]\!, \end{equation}
where the second set of summations is replaced using the symmetry in (\ref{eq:symmetricweights}).

On close examination, there is a lot of redundant information in $G^g$. One really doesn't need the mirror copy of $G$, since the
information attached to the adjoint nodes can be recorded associated to the original nodes and the arc weights of the mirror arcs
are the same. Now if we fold back the mirror copy over the original, identifying nodes $k$ and $\overline{k}$, we obtain a graph
with same node set as $G$ but with an enlarged set of arcs. Arcs in $A$ will be replaced by pairs of arcs in opposite directions
and nonlinear arcs will become either loops (in case one had an arc $(i,\overline{\textit{\i}})$ in $A_3$) or pairs of arcs with
opposite orientations between the same pair of nodes. Also, equations (\ref{eq:symmetricweights}) and (\ref{eq:pesosnaolineares})
imply that all arcs in parallel have the same weight, see Figure~\ref{fig:foldgradientcomputationalgraph}. This is still too much
redundancy. We may leave arcs in $A$ as they are, replace the pairs of directed nonlinear arcs in opposite directions with a
single undirected arc between the same pair of nodes and the directed loops by undirected ones, as exemplified in
Figure~\ref{fig:foldgradientcomputationalgraph}, as long as we keep in mind the special characteristics of the paths we're
interested in.

 \begin{figure}
\begin{center}
\includegraphics[scale=.82
]{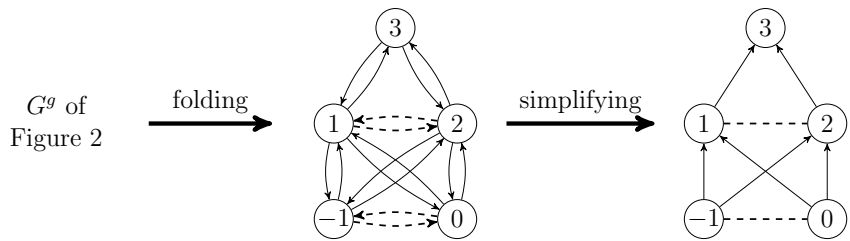}
\caption{Folding of the gradient computational graph of Figure~\protect\ref{fig:gradientcomputationalgraph}
and further elimination of redundancies.\label{fig:foldgradientcomputationalgraph}}
 \end{center}
\end{figure}

The paths needed for the computation of the Hessian, in the folded and simplified graph,  are divided into three parts. In the
first part we have a directed path from some zero in-degree node, say $i$, to some other node, say $r$. Next comes an undirected
nonlinear arc $(r,s)$. The last part is a path from $s$ to another zero in-degree node, say $j$, in which all arcs are traveled in
the wrong direction. Of course, both the first and third parts of the path may be empty, only the middle part (the nonlinear arc)
is mandatory.

This folded and simplified graph can be interpreted as a reduced gradient graph, with the symmetric redundancies removed. The
graph together with the tri-parted path interpretation for partial derivatives constitutes our graph model for the Hessian. In
Section~\ref{sec:ep} we will present an algorithm that takes full advantage of \revisedtext{these} symmetries and has a natural
interpretation as an algorithm that gradually introduces nonlinear arcs and accumulates the weights of these special paths in the
computational graph.
In contrast, the authors in~\cite{Bhowmick:2008} build the entire gradient graph to then
use an axial symmetry detection algorithm on this computational
graph in order to eliminate redundancies. Once this is done, the Hessian is calculated via Jacobian methods.

\section{Hessian formula}\label{sec:HessianFormula}

The closed formula to be developed concerns the Hessian of a function $g$ that is defined as a linear combination of the functions
 $\Psi_1$, \ldots, $\Psi_p$, or, in matrix form,
\begin{equation}\label{def:g} g(x) = y^T \Psi (x),\end{equation}
where $y\in \R^p$ and $\Psi\in C^2(\R^n,\R^p)$. The linearity of the differential operator implies that the Hessian of $g$ is
simply the linear combination of the Hessians of $\Psi_1$, \ldots, $\Psi_p$:
\begin{equation}\label{def:hessg}g''(x) = \sum_{i=1}^p y_i \Psi''_i(x).\end{equation}
This motivates the introduction of the following definition of the vector-tensor product $y^T \Psi''(x)$, in order to establish an
analogy between the linear combinations in (\ref{def:g}) and (\ref{def:hessg}):
\begin{equation}\label{def:vectortensorproduct}
g''(x) = (y^T\Psi(x))'' = y^T \Psi'' = \sum_{i=1}^p y_i \Psi''_i(x).
\end{equation}

Next we need to establish how to express $g''$ when $\Psi$ is a composition of vector functions of several variables, the subject
of the next Proposition.

\begin{proposition}\label{prop:hesscomp} Let $y\in \R^p$, $\Omega \in C^2(\R^n,\R^m)$, $\Theta \in C^2(\R^m,\R^p)$ and $\Psi (x)=
\Theta \circ
\Omega(x)$. Then
\begin{equation}\label{eq:hesscomp}
y^T \Psi'' = (\Omega')^T (y^T \Theta'')\Omega'+ (y^T\Theta') \Omega''.
\end{equation}
\end{proposition}

\begin{proof} By definition, applying differentiation rules, and using the symmetry of the Hessian, we may calculate entry $(j,k)$ of the
Hessian as follows:
\begin{eqnarray*}
 (y^T \Psi''(x))_{jk} &=& \sum_i y_i\; \frac{\partial^2 \Psi_i(x)}{\partial x_j \partial x_k}\\
 &=& \sum_i y_i\; \frac{\partial }{\partial x_j}\left(\frac{\partial \Theta_i(\Omega(x))}{\partial x_k}\right)\\
 &=&  \sum_i y_i\; \frac{\partial}{\partial x_j}\left(\sum_{r=1}^m\frac{\partial \Theta_i(\Omega(x))}{\partial
\Omega_r}\frac{\partial \Omega_r(x)}{\partial x_k}\right)\\
 &=& \displaystyle\sum_i \sum_r y_i \left[\frac{\partial}{\partial x_j}\left(\frac{\partial \Theta_i(\Omega(x))}{\partial
\Omega_r}\right)\right]\frac{\partial \Omega_r(x)}{\partial x_k} +\sum_i \sum_r y_i\; \frac{\partial \Theta_i(\Omega(x))}{\partial
\Omega_r} \frac{\partial^2 \Omega_r(x)}{\partial x_j\partial x_k}\\
 &=& \sum_r  \sum_s \sum_i y_i\; \frac{\partial^2\Theta_i(\Omega(x))}{\partial \Omega_s\partial \Omega_r}\frac{\partial
\Omega_s(x)}{\partial x_j} \frac{\partial \Omega_r(x)}{\partial x_k}+ \sum_r (y^T\Theta'(\Omega(x))_r (\Omega''_r(x))_{jk}\\
 &=& \sum_s \sum_r (y^T \Theta''(\Omega(x)))_{rs} \left(\Omega'(x)\right)_{sj} \left(\Omega'(x)\right)_{rk}
 + \sum_r (y^T\Theta'(\Omega(x))_r (\Omega''_r(x))_{jk}\\
 &=&\sum_s  \left(\Omega'(x)\right)_{sj} \sum_r (y^T \Theta''(\Omega(x)))_{sr}  \left(\Omega'(x)\right)_{rk}+ \sum_r
(y^T\Theta'(\Omega(x))_r (\Omega''_r(x))_{jk},\\
 &=&\left((\Omega'(x))^T(y^T\Theta''(\Omega(x)))\Omega'(x)\right)_{jk} + \left((y^T\Theta'(\Omega(x))) \Omega''(x)\right)_{jk},
 \end{eqnarray*}
which is the entry $(j,k)$ of the right-hand-side of (\ref{eq:hesscomp}).\end{proof}

Although we want to express the Hessian of a composition of state transformations, it is actually easier to obtain the closed form
for the composition of generic vector multivariable functions, our next result.

\begin{proposition}\label{prop:hessgencomp}
Let $\Psi_i(x)\in C^2(\R^{m_{i-1}},\R^{\revisedtext{m_i}})$, for $i=1,\ldots, k$, $y\in \R^{m_k}$ and
\[ g(x) = y^T \left( \Psi_k\circ\cdots \circ \Psi_1\right)(x).\]
Then
\begin{equation}\label{eq:hessgeneral}
g'' = \sum_{i=1}^k \left(\prod_{j=1}^{i-1}(\Psi'_j)^T \right)((\overline{\mathrm{w}}^i)^T\Psi''_i) 
\left(\prod_{j=1}^{i-1}\Psi'_{i-j} \right),
\end{equation}
where
\begin{equation}\label{eq:w}
(\overline{\mathrm{w}}^i)^T = y^T\prod_{j=1}^{k-i}\Psi'_{k-j+1},\quad \mbox{for }i=1,\ldots,k.
\end{equation}
\end{proposition}

\begin{proof} The proof is by induction on $k$. When $k=1$, the result is trivially true, since in this case
(\ref{eq:hessgeneral})--(\ref{eq:w}) reduce to $(\overline{\mathrm{w}}^1)^T \Psi''_1= y^T \Psi''_1$, which denotes, according to
(\ref{def:hessg}), the Hessian of $g$.

Assume the proposition is true when $g$ is the composition of $k-1$ functions. Now simply rewrite the composition of $k$ functions
as follows
\begin{equation}\label{eq:k-1} g = y^T \Psi_k \circ \cdots \circ \Psi_3 \circ \Psi,\end{equation}
where $\Psi = \Psi_2 \circ \Psi_1$. Then, applying the induction hypothesis to (\ref{eq:k-1}), we obtain
\begin{equation}\label{eq:hessk-1}
g'' = (\Psi')^T \left[\sum_{i=3}^k  \left(\prod_{j=3}^{i-1}(\Psi'_j)^T \right)((\overline{\mathrm{w}}^i)^T\Psi''_i) 
\left(\prod_{j=3}^{i-1}\Psi'_{i-j} \right)\right] \Psi' + (\overline{\mathrm{w}}^2)^T \Psi'',
\end{equation}
where $(\overline{\mathrm{w}}^2)^T \Psi''$ is simply the first summand of (\ref{eq:hessgeneral}), associated with
the rightmost function in (\ref{eq:k-1}).

The last term in (\ref{eq:hessk-1}) is calculated separately, using \revisedtext{Proposition~\ref{prop:hesscomp}} and
(\ref{eq:w}):
\begin{eqnarray}
(\overline{\mathrm{w}}^2)^T \Psi'' &=& (\Psi_1')^T ((\overline{\mathrm{w}}^2)^T \Psi_2'')\Psi_1'+
((\overline{\mathrm{w}}^2)^T\Psi_2') \Psi_1'' \nonumber\\
&=&(\Psi_1')^T ((\overline{\mathrm{w}}^2)^T \Psi_2'')\Psi_1'+ (\overline{\mathrm{w}}^1)^T \Psi_1''.\label{eq:lastterm}
\end{eqnarray}
Using the fact that $\Psi'= \Psi_2'\Psi_1'$, and expression (\ref{eq:lastterm}) obtained for the last term, (\ref{eq:hessk-1})
becomes
\begin{eqnarray*}
g'' &= &  (\Psi'_1)^T(\Psi'_2)^T \left[\sum_{i=3}^k  \left(\prod_{j=3}^{i-1}(\Psi'_j)^T
\right)((\overline{\mathrm{w}}^i)^T\Psi''_i) 
\left(\prod_{j=3}^{i-1}\Psi'_{i-j} \right)\right] \Psi_2'\Psi_1'\\
&&+(\Psi_1')^T ((\overline{\mathrm{w}}^2)^T \Psi_2'')\Psi_1'+ (\overline{\mathrm{w}}^1)^T \Psi_1'' \\
&=& \sum_{i=1}^k  \left(\prod_{j=1}^{i-1}(\Psi'_j)^T \right)((\overline{\mathrm{w}}^i)^T\Psi''_i) 
\left(\prod_{j=1}^{i-1}\Psi'_{i-j} \right),
\end{eqnarray*}
which completes the proof.\end{proof}

The Hessian of the composition of state transformations follows easily from Proposition~\ref{prop:hessgencomp}.

\begin{corollary}\label{cor:hessf}
Let $f$ be the composition of state transformations given in (\ref{eq:StateTransform}). Then its Hessian is
\begin{equation}\label{eq:hessf}
f'' =  P \sum_{i=1}^\ell \left(\prod_{j=1}^{i-1}(\Phi'_j)^T \right)((\overline{\mathrm{v}}^i)^T\Phi''_i) 
\left(\prod_{j=1}^{i-1}\Phi'_{i-j} \right) P^T,
\end{equation}
where
\begin{equation}\label{eq:v}
(\overline{\mathrm{v}}^i)^T = e^T_\ell\prod_{j=1}^{\ell-i}\Phi'_{\ell-j+1},\quad \mbox{for }i=1,\ldots,\ell.
\end{equation}
\end{corollary}
\begin{proof} Simply apply (\ref{eq:hessgeneral}) to the composition of $\ell+1$ functions, where $\Psi_i=\Phi_i$, for $i=1,\ldots, \ell$,
$\Psi_0(x)=P^T x$, and use the facts that $\Psi_0' = P^T$ and $\Psi_0''=0$.\end{proof}

As an application of (\ref{eq:hessf}), we have used it in \cite{GowerMelloTechReportSept2010} to show the
correctness of Griewank and Walther's reverse Hessian computation algorithm \cite[p. 157]{Griewank:2008}.
A number of other methods are also demonstrated using~(\ref{eq:hessf}), such as the forward Hessian mode, reverse
Hessian-vector products and a novel forward mode in~\cite{Gowermaster}.

\section{A new Hessian computation algorithm: \texttt{edge\_pushing}}\label{sec:ep}

\subsection{Development}

In order to arrive at an algorithm to efficiently compute expression (\ref{eq:hessf}), it is helpful to think in terms of block operations. First of all, rewrite (\ref{eq:hessf}) as
\begin{equation}\label{eq:hessfexpanded1}
f'' = P W P^T=P\left(\sum_{i=1}^\ell W_i\right) P^T,\end{equation}
so the problem boils down to the computation of $W$. The summands in $W$ share a common structure, spelled out below for the $i$-th summand.
\begin{equation}\label{eq:structurei}
W_i=\underbrace{\left((\Phi'_1)^T\cdots (\Phi'_{i-1})^T\right)}_{\mbox{left multiplicand}}
\underbrace{\left((\overline{\mathrm{v}}^i)^T\Phi''_i\right) }_{\mbox{central multiplicand}}
\underbrace{\left(\Phi'_{i-1}\cdots \Phi'_{1}\right).}_{\mbox{ right multiplicand}}\end{equation}

Using the distributivity of multiplication over addition, the partial sum $W_i+W_{i-1}$ may be expressed as a three multiplicand
product\revisedtext{s} where the left and right multiplicands coincide with those in the expression of $W_{i-1}$, but the
central one is different.
\begin{eqnarray}
\lefteqn{W_i+W_{i-1} =} \nonumber\\
&&\kern-5pt \left((\Phi'_1)^T\cdots (\Phi'_{i-2})^T\right)\!\left((\Phi'_{i-1})^T((\overline{\mathrm{v}}^i)^T\Phi''_i)(\Phi'_{i-1}) +((\overline{\mathrm{v}}^{i-1})^T\Phi''_{i-1}) \right)\!\left(\Phi'_{i-2}\cdots \Phi'_{1}\right)\!.~ \label{eq:Wi+Wi+1}
\end{eqnarray}
Instead of calculating each $W_i$ separately, we may save effort by applying this idea to increasing sets of partial sums, all
the way to $W_\ell$. This alternative way of calculating $W$ is reminiscent of \revisedtext{Horner's scheme, that uses nesting
to efficiently calculate polynomials~\cite{Horner:1819}}.

 The nested expression for $\ell=3$ is given in
(\ref{eq:nested}) below.
\begin{equation}\label{eq:nested}
W=(\Phi'_1)^T\revisedtext{\left[}(\Phi'_2)^T ((\overline{\mathrm{v}}^{3})^T\Phi''_{3})\Phi'_2 +
(\overline{\mathrm{v}}^{2})^T\Phi''_{2}\revisedtext{\right]}
  \Phi'_1 + (\overline{\mathrm{v}}^{1})^T\Phi''_{1}.\end{equation}

Of course, the calculation of such a nested expression must begin at the innermost expression and proceed outwards. This means,
in this case, going from the highest to the lowest index. This is naturally accomplished in a backward sweep of the computational
graph, which could be schematically described as follows.
\begin{eqnarray*}
\mbox{Node $\ell$}&& W\leftarrow (\overline{\mathrm{v}}^\ell)^T \Phi_\ell''\\
\mbox{Node $\ell-1$} && W\leftarrow (\Phi_{\ell-1}')^T W \Phi_{\ell-1}'\\
 && W\leftarrow W+(\overline{\mathrm{v}}^{\ell-1})^T \Phi_{\ell-1}''\\
 &\vdots &\\
 \mbox{Node $i$} &&W\leftarrow (\Phi_i')^T W \Phi_i'\\
 && W\leftarrow W+(\overline{\mathrm{v}}^i)^T \Phi_i''\\
 &\vdots &\\
 \mbox{Node $1$} &&W\leftarrow (\Phi_1')^T W \Phi_1'\\
 && W\leftarrow W+(\overline{\mathrm{v}}^1)^T \Phi_1''. \end{eqnarray*}
In particular, node $\ell$'s iteration may be cast in the same format as the other ones if we initialize $W$ as a null matrix.

It follows that
\normalcolor
the value of $W$ at the end of the iteration where node $i$ is swept is given by
\[W = \sum_{k=i}^\ell \left(\prod_{j=i}^{k-1}(\Phi'_j)^T \right)((\overline{\mathrm{v}}^k)^T\Phi''_k) 
\left(\prod_{j=1}^{k-i}\Phi'_{k-j} \right).\]
Notice that, at the iteration where node $i$ is swept, both assignments involve derivatives of $\Phi_i$, which are available. The
other piece of information needed is the vector $\overline{\mathrm{v}}^i$, which we know how to calculate via a backward sweep
from Algorithm~1. Putting these two together, we arrive at Algorithm~3.

\begin{center}
\parbox{.6\textwidth}{%
\includegraphics[scale=.83
]{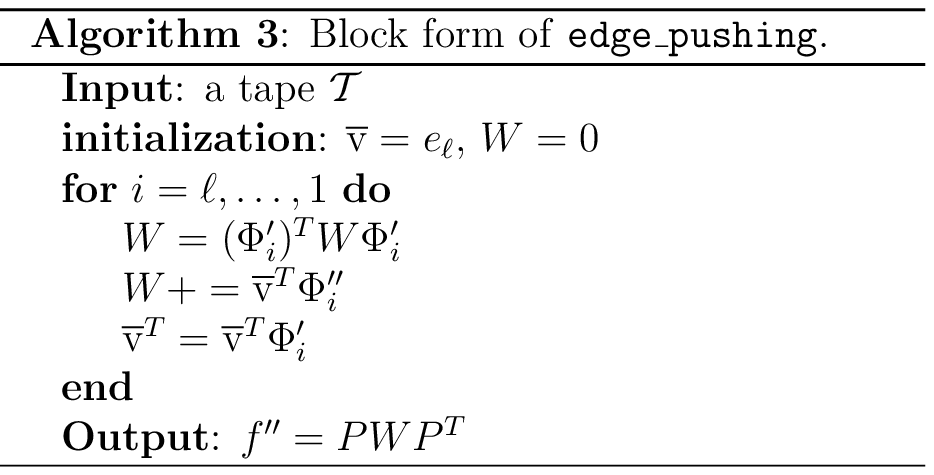}}
\end{center}

Before delving into the componentwise version of Algorithm~3, there is a key observation to be made about
matrix $W$, established in the following proposition.

\begin{proposition}
At the end of the iteration at which node $i$ is swept in Algorithm~3, for all $i$, the nonnull elements
of $W$ lie in the upper diagonal block of size $n+i-1$.
\end{proposition}

\begin{proof} Consider the first iteration, at which node $\ell$ is swept. At the beginning $W$ is null, so the first block assignment
($(\Phi_\ell')^T W \Phi_\ell'$) does not change that. Now consider the assignment 
\[W\leftarrow W+(\overline{\revisedtext{\mathrm{v}}})^T \Phi_\ell''.\]
Using (\ref{def:vectortensorproduct}) and the initialization of $\overline{\mathrm{v}}$, we have
\[ (\overline{\mathrm{v}})^T \Phi_\ell'' = \bar{v}_\ell [\Phi_\ell]_\ell'' = [\Phi_\ell]_\ell'' ,\]
and, since $[\Phi_\ell]_\ell (y) = \phi_\ell (y_j)_{j\prec \ell}$, the nonnull entries of $[\Phi_\ell]_\ell''$ must have column
and row indices that correspond to predecessors of node $\ell$. This means the last row and column, of index $\ell$, are zero.
Thus the statement of the proposition holds after the first iteration.

Suppose by induction that, after node $i+1$ is swept, the last $\ell-i$ rows and columns of $W$ are null. Recalling
(\ref{eq:PhiJacobian}) and using the induction hypothesis, the matrix-product $(\Phi_i')^T W \Phi_i'$ can be written in block form
as follows:
\[
\kern-8pt\left[ \begin{array}{@{\hspace{7pt}}c@{\hspace{8pt}}|@{\hspace{5pt}}c@{\hspace{3pt}}|@{\hspace{11pt}}c@{\hspace{7pt}}} 
\rule{0pt}{23pt}\rule[-13pt]{0pt}{10pt}I & c^{i} & 0\\ 
\hline
0\rule{0pt}{12pt} & 0 & 0 \\ \hline
\rule{0pt}{23pt}\rule[-10pt]{0pt}{10pt}0 & 0 & I\end{array}\right]
\left[ \begin{array}{@{\hspace{7pt}}c@{\hspace{8pt}}|@{\hspace{5pt}}c@{\hspace{5pt}}|@{\hspace{11pt}}c@{\hspace{7pt}}} 
\rule{0pt}{23pt}\rule[-13pt]{0pt}{10pt}W_{1-n..i-1,1-n..i-1} & W_{1-n..i-1,i} &0\\ 
\hline
W_{i,1-n..i-1}\rule{0pt}{12pt} &  w_{ii} & 0 \\ \hline
\rule{0pt}{23pt}\rule[-10pt]{0pt}{10pt}0 & 0 & 0\end{array}\right]
\left[ \begin{array}{@{\hspace{5pt}}c@{\hspace{2pt}}|@{\hspace{5pt}}c@{\hspace{5pt}}|@{\hspace{11pt}}c@{\hspace{7pt}}} 
\rule{0pt}{23pt}\rule[-13pt]{0pt}{10pt}I & 0 & 0\makebox(0,10)[lc]{\footnotesize$\begin{array}{@{\hspace{.5cm}}c}1-n\\[-5pt]
\vdots \\[-3pt] i-1\end{array}$}\\ 
\hline
(c^i)^T\rule{0pt}{12pt} & 0 & 0\makebox(0,10)[lc]{\footnotesize\hspace{.5cm}row $i$,} \\ \hline
\rule{0pt}{23pt}\rule[-10pt]{0pt}{10pt}0 & 0 & I\makebox(0,10)[lc]{\footnotesize$\begin{array}{@{\hspace{.5cm}}c}i+1\\[-5pt]
\vdots \\[-3pt] \ell\end{array}$}\end{array}\right]\]
which results in
\begin{equation}\label{eq:Wblock}
\left[ \begin{array}{@{\hspace{2pt}}c@{\hspace{5pt}}|@{\hspace{5pt}}c@{\hspace{5pt}}|@{\hspace{11pt}}c@{\hspace{5pt}}} 
\rule{0pt}{23pt}\rule[-13pt]{0pt}{10pt}W_{1-n..i-1,1-n..i-1}+ c^{i}W_{i,1-n..i-1}+W_{1-n..i-1,i}\,
(c^i)^T\kern-2pt+ w_{ii}\,c^{i} (c^i)^T & 0 &
0\makebox(0,10)[lc]{\footnotesize$\begin{array}{@{\hspace{.5cm}}c}1-n\\[-5pt] \vdots \\[-3pt] i-1\end{array}$}\\ 
\hline
0 & 0 & 0\makebox(0,10)[lc]{\footnotesize\hspace{.5cm}row $i$.} \\ \hline
\rule{0pt}{23pt}\rule[-10pt]{0pt}{10pt}0 & 0 & 0\makebox(0,10)[lc]{\footnotesize$\begin{array}{@{\hspace{.5cm}}c}i+1\\[-5pt]
\vdots \\[-3pt] \ell\end{array}$}\end{array}\right]\rule{15pt}{0pt}
\end{equation}
Thus at this point the last $\ell-(i-1)$ rows and columns have been zeroed. 

Again using (\ref{def:vectortensorproduct}), we have
\[ (\overline{\mathrm{v}})^T \Phi_i'' = \bar{v}_i\left(\frac{\partial^2\phi_i}{\partial v_j\partial v_k}\right)_{1-n\leq j,k\leq \ell} ,\]
where the nonnull entries of the Hessian matrix on the right-hand side have column and row indices that correspond to predecessors of node $i$. Therefore,
the last $\ell-(i-1)$ rows and columns of this Hessian are also null. Hence the second and last block assignment involving $W$
will preserve this property, which, by induction, is valid till the end of the algorithm.\end{proof}

Using the definition of $c^{i}$ in (\ref{eq:definicaocT}), the componentwise translation in the first block
assignment involving $W$
in Algorithm~3 is
\begin{equation} \label{eq:componentpush}
 \left((\Phi_i ')^T W \Phi_i ' \right)_{jk} =
\left\{
\begin{array}{ll}\displaystyle
 w_{jk} + \frac{\partial \phi_i}{ \partial v_k}\frac{\partial \phi_i}{ \partial v_j }w_{ii} +
	\frac{\partial \phi_i}{ \partial v_k}w_{ji} + \frac{\partial \phi_i}{ \partial v_j }w_{ik},\quad  &\mbox{if } j<i \mbox{
and } k < i,\\
0,         &\mbox{otherwise.} \\
\end{array}
\right.
\end{equation}

For the second block assignment, using (\ref{def:vectortensorproduct}), we have that
\begin{equation}\label{eq:componentcreation}
\left((\overline{\mathrm{v}})^T \Phi_i'' \right)_{jk} = \left\{ \begin{array}{ll}\bar{v}_i \dfrac{\partial^2 \phi_i}{\partial v_j\partial
v_k},\quad &\mbox{if $j<i$ and $k<i$,}\\ 0,&\mbox{otherwise.}\end{array}\right.\end{equation}

Finally, notice that, since the componentwise version of the block assignment, done as node $i$ is swept, involves only entries
with row and column indices smaller than or equal to $i$, one does not need to actually zero out the row and column $i$ of $W$, as
these entries will not be used in the following iterations.

This componentwise assignment may be still simplified using symmetry, since $W$'s symmetry is preserved throughout
\texttt{edge\_pushing}. In order to avoid unnecessary calculations with symmetric counterparts, we employ the notation
$w_{\{ji\}}$ to denote both $w_{ij}$ and $w_{ji}$. Notice, however, that, when $j=k$ in (\ref{eq:componentpush}), we have
\[ \left((\Phi_i ')^T W \Phi_i ' \right)_{jj} = 
w_{jj} +\left( \frac{\partial \phi_i}{\partial v_j}\right)^2w_{ii} +  \frac{\partial \phi_i}{ \partial v_j }w_{ji} + 
\frac{\partial \phi_i}{ \partial v_j }w_{ij},\]
so in the new notation we would have
\[\left((\Phi_i ')^T W \Phi_i ' \right)_{\{jj\}} = w_{\{jj\}} +\left( \frac{\partial \phi_i}{\partial v_j}\right)^2w_{\{ii\}} + 
2\frac{\partial \phi_i}{ \partial v_j }w_{\{ji\}}.\]

The componentwise version of Algorithm~3 adopts the point of view of the node being swept. Say, for
instance that node $i$ is being swept. Consider the first block assignment
\[W\leftarrow (\Phi_i')^T W \Phi_i',\]
whose componentwise version is given in (\ref{eq:componentpush}).
Instead of focusing on updating each $w_{\{jk\}}$, $j,k < i$, at once, which would involve accessing $w_{\{ii\}}$, $w_{\{ji\}}$
and $w_{\{ik\}}$, we focus on each $w_{\{pi\}}$ at a time, and `push' its contribution to the appropriate $w_{\{jk\}}$'s. Taking
into account that the partial derivatives of $\phi_i$ may only be nonnull with respect to $i$'s predecessors, these appropriate
elements will be $w_{\{jp\}}$, where $j\prec i$, see the \texttt{pushing} step in Algorithm~4.

The second block assignment
\[W\leftarrow W+(\overline{\mathrm{v}}^i)^T \Phi_i'',\]
may be thought of as the creation of new contributions, that are added to appropriate entries and that will be pushed in later
iterations. From its componentwise version in (\ref{eq:componentcreation}), we see that only entries of $W$ associated with
predecessors of node $i$ may be changed in this step. The resulting componentwise version of the \texttt{edge\_pushing} algorithm
is Algorithm~4.
\begin{figure}
\centering
\includegraphics[scale=.83
]{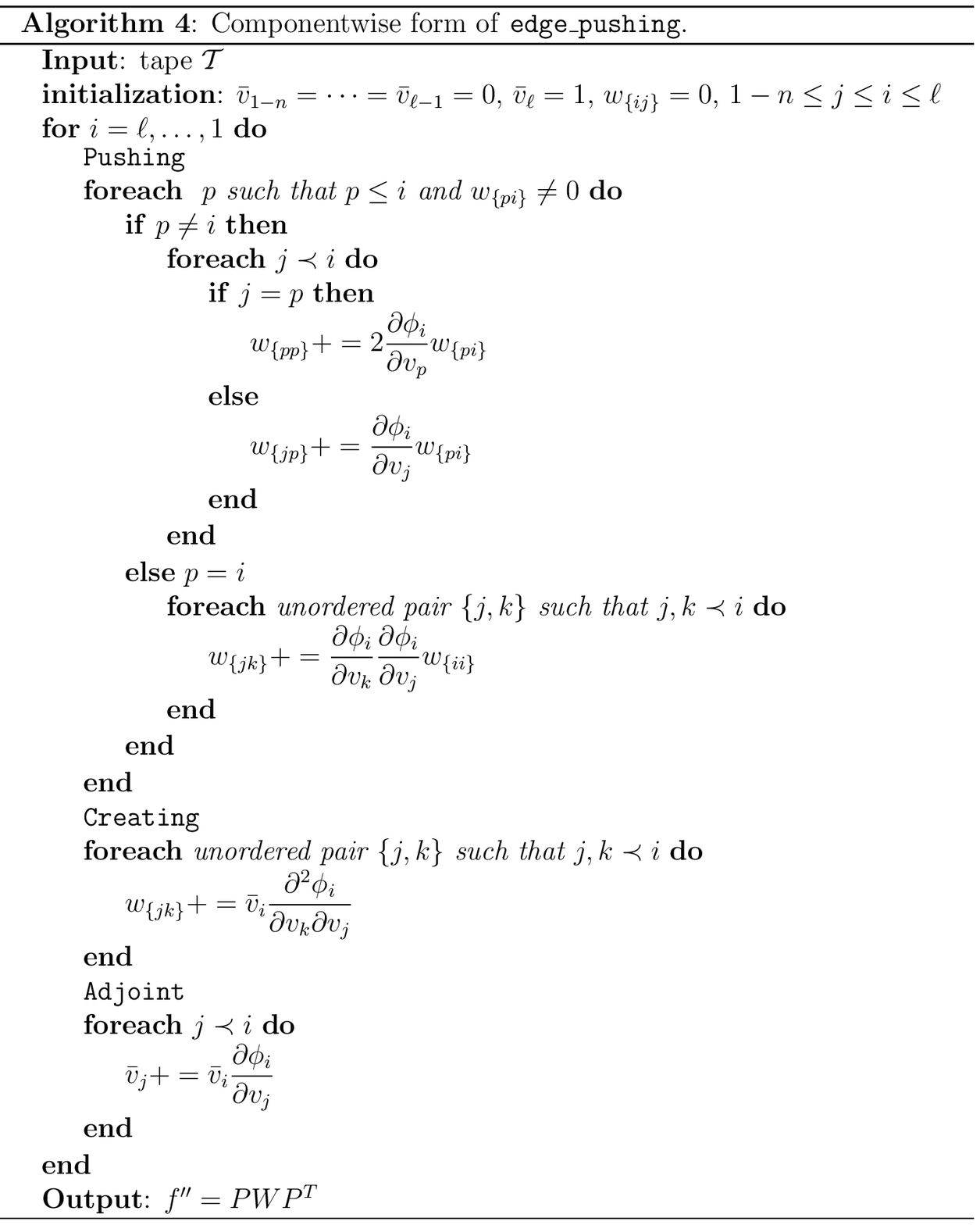}
\end{figure}

Algorithm~4 has a very natural interpretation in terms of the graph model introduced in Section~\ref{sec:HessianGraphModel}. The
nonlinear arcs are `created' and their weight initialized (or updated, if in fact they already exist) in the \texttt{creating}
step. In graph terms, the \texttt{pushing} step performed when node $i$ is swept actually pushes the endpoints of the nonlinear
arcs incident to node $i$ to its predecessors. The idea is that subpaths containing the nonlinear arc are replaced by shortcuts.
This follows from the fact that if a path contains the nonlinear arc $\{i,p\}$, then it must also contain precisely one of the
other arcs incident to node $i$. Figure~\ref{fig:shortcuts} illustrates the possible subpaths and corresponding shortcuts. In
cases I and III, the subpaths consist of two arcs, whereas in case \revisedtext{II}, three arcs are replaced by a new nonlinear
arc. Notice
that the endpoints of a loop (case II) may be pushed together down the same node, or split down different nodes.  In this way, the
contribution of each nonlinear arcs trickles down the graph, distancing the higher numbered nodes until it finally reaches the
independent nodes.

\begin{figure}
     \centering
\includegraphics[scale=.82
]{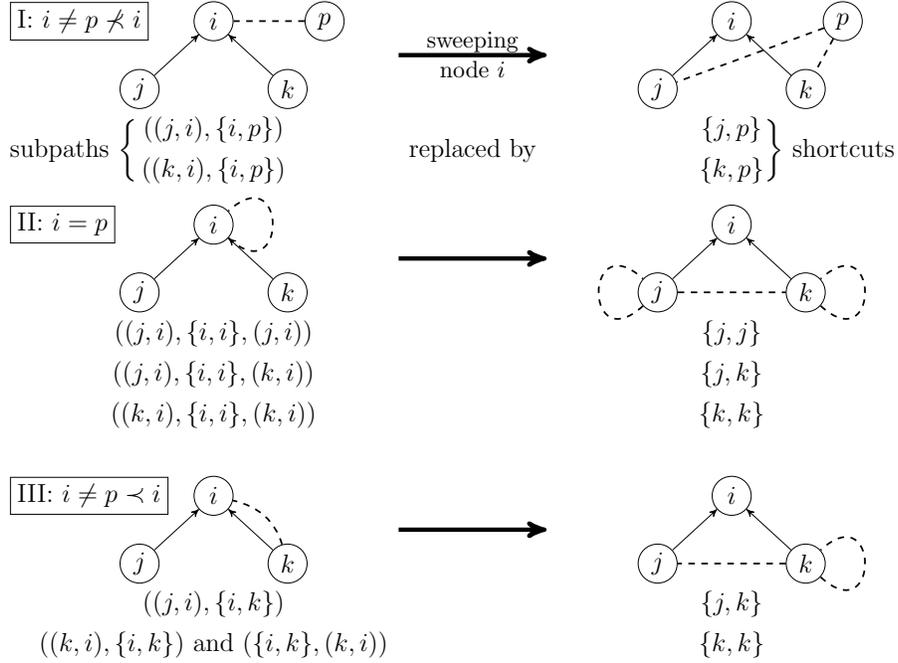}
  \caption{Pushing nonlinear arc $\{i,p\}$ is creating shortcuts.} \label{fig:shortcuts}
\end{figure}

This interpretation helps in understanding the good performance of \texttt{edge\_pushing} in the computational tests, in the sense
that only ``proven'' contributions to the Hessian (nonlinear arcs) are dealt with.

\subsection{Example}

In this section we run Algorithm~4 on one example, to better illustrate its workings. Since we're doing it on paper, we have the
luxury of doing it symbolically.

The iterations of \texttt{edge\_pushing} on a computational graph of the function $f(x)=(x_{-2} + e^{x_{-1}})(3x_{-1}+x_0^2)$ are
shown on Figure~\ref{fig:edgepushingexample}. The thick arrows indicate the sequence of three iterations. Nodes about to be swept
are highlighted. As we proceed to the graph on the right of the arrow, nonlinear arcs are created (or updated), weights are
appended to edges and adjoint values are updated, except for the independent nodes, since the focus is not gradient computation.
For instance, when node 3 is swept, the nonlinear arc $\{1,2\}$ is created. This nonlinear arc is pushed and split into two when
node 2 is swept, becoming nonlinear arcs $\{0,1\}$ and $\{-1,1\}$, with weights $1\cdot 2 v_0$ and $1\cdot 3$, respectively. When
node 1 is swept, the nonlinear arc $\{-1,1\}$ is pushed and split into nonlinear arcs $\{-2,-1\}$ and $\{-1,-1\}$, the latter with
weight $2\cdot 3\cdot e^{v_{-1}}$. Later on, in the same iteration, the nonlinear contribution of node 1,
$\partial^2\phi_1/\partial v_{-1}^2$, is added to the nonlinear arc $\{-1,-1\}$. Other operations are analogous. The Hessian can
be retrieved from the weights of the nonlinear arcs between independent nodes at the end of the algorithm:
\[f''(x) = \left(\begin{array}{*{3}{c}} 0 & 3 & 2 v_0\\
3 & e^{v_{-1}}(6+v_2) & 2 v_0 e^{v_{-1}}\\
 2 v_0 & 2 v_0 e^{v_{-1}} & 2 v_1\end{array}\right) = 
\left(\begin{array}{*{3}{c}}0 & 3 & 2 x_0\\
3 & e^{x_{-1}} ( 6 + 3 x_{-1} + x_0^2) & 2 x_0 e^{x_{-1}} \\ 
2x_0 & 2 x_0 e^{x_{-1}} & 2(x_{-2}+e^{x_{-1}}) \end{array}\right).\]
Notice that arcs that are pushed are deleted from the figure just for clarity purposes, though this is not explicitly done in
Algorithm~4. Nevertheless, in the actual implementation the memory locations corresponding to these arcs are indeed deleted, or,
in other words, made available, since this can be done in constant time.

\begin{figure}
\begin{center}
\rule{1.5cm}{0pt}
\includegraphics[scale=.82
]{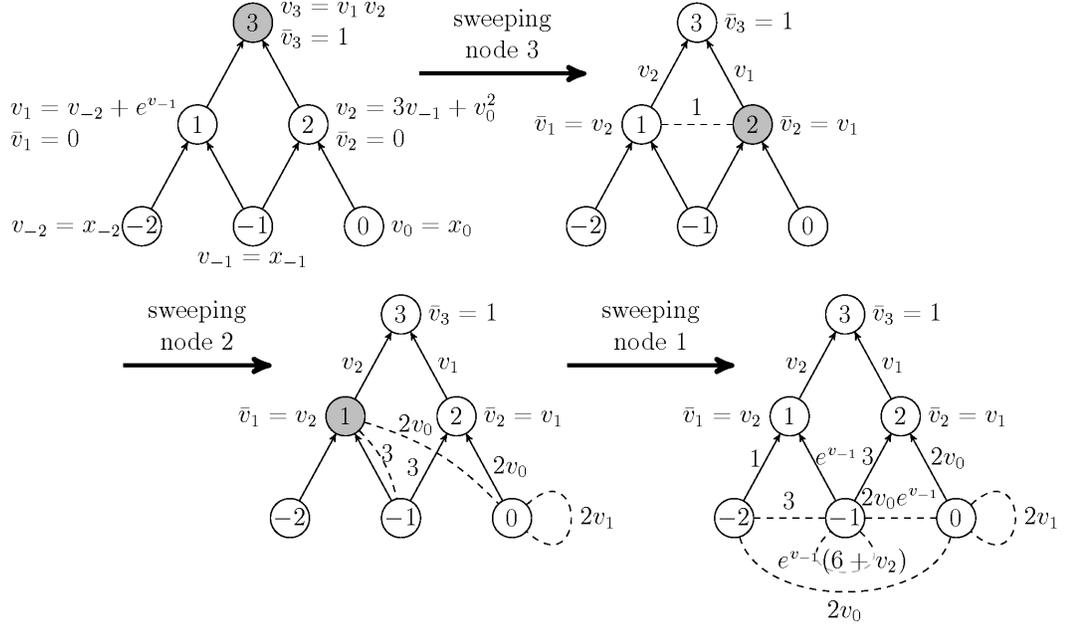}
\caption{\texttt{edge\_pushing} applied to a computational graph of $f(x) = (x_{-2} +
e^{x_{-1}})(3x_{-1}+x_0^2)$.}\label{fig:edgepushingexample}
\end{center}
\end{figure}

\subsection{\texttt{edge\_pushing} complexity bounds}

For our bounds we assume that the data structure used for $W$ in Algorithm~4 is an adjacency list. This is a structure appropriate
for large sparse graphs, which shall be our model for $W$, denoted by $G_W$. The entries in $W$ are interpreted as the set of arc
weights. Thus the nodes of $G_W$ are associated with the rows of $W$. Notice that this is the same as the set of nodes of the
computational graph. The support of $W$ is associated to the set of arcs of $G_W$. During the execution of the algorithm, new arcs
may be created during the \texttt{pushing} or the \texttt{creating} step. After node $i$ has been swept, $G_W$ has accumulated all
the nonlinear arcs that have been created or pushed, up to this iteration, since arcs are not deleted. On may think of $G_W$ as
the recorded history (creation and pushing) of the nonlinear arcs.

Denote by $N_i$ the set of neighbors of node $i$ in $G_W$ and by $d_i$ the degree of  node $i$. Of course the degree of node $i$
and  its neighborhood vary during the execution of the algorithm. The time for inserting or finding an arc $\{i,j\}$ and its
weight $w_{\{i,j\}}$ is bounded by $O(d_i+ d_j)$, where $d_i$ and $d_j$ are the degrees at the iteration where the operation takes
place. We assume that the set of elemental functions is composed of only unary and binary functions.

\subsubsection{Time complexity}

The time complexity of \texttt{edge\_pushing} depends on how many nonlinear arcs are allocated during execution. Thus it is
important to establish bounds for the number of arcs allocated to each node. Furthermore, we may fix $G_W^*$ as the graph obtained
at the end of the algorithm.

Let $d^*_i$ be the degree of node $i$ in $G_W^*$, and let $d^* = \max_i\{d_i^*\}$. Clearly $d_i\leq d^*_i$, where $d_i$ is the
degree of node $i$ in the graph $G_W$ at any given iteration. In order to bound the complexity of \texttt{edge\_pushing}, we
consider the \texttt{pushing} and \texttt{creating} steps separately. 

Studying the cases spelled out in Figure~\ref{fig:shortcuts}, one concludes that the time spent in pushing edge $\{i,p\}$ is
bounded by, in each case: \\[8pt]
\centerline{\begin{tabular}{l|l} 
\hline
Case & Upper bound for time spent\\ \hline
I: $i \neq p \nprec i $ & $d_j^* + 2d_p^* + d_k^*$\\
II: $i=p$ & $d_j^* + 2d_k^* $\\
III: $i \neq p \prec i$ & $d_j^* + 2d_k^* $ \\ \hline
\end{tabular}}\\[8pt]
Hence $2(d_j^* + d_p^* + d_k^*)$ is a common bound, where $j$ and $k$ are predecessors of node $i$. Since there are at most
$d_i^*$ nonlinear arcs incident to node $i$, the time spent in the \texttt{pushing} step at the iteration where node $i$ is swept
is bounded by
\[ d_i^*(2(d_j^* + d_p^* + d_k^*)) = O(d_i^*(d_j^* + d_p^* + d_k^*))=O(d_i^*d^*).\]

Finally, the assumption that all functions are either unary or binary implies that at most three nonlinear arcs are allocated
during the \texttt{creating} step, for each iteration of \texttt{edge\_pushing}. Hence the time used up in this step at the
iteration where node $i$ is swept is bounded by
\[2(d_j^*+d_k^*) = O(d_j^*+d_k^*) = O(d^*),\]
where $j$ and $k$ are predecessors of node $i$.

Thus, taking into account the time spent in merely visiting a node --- say, when the intermediate function associated with the
node is linear --- is constant, the time complexity of \texttt{edge\_pushing} is
\begin{eqnarray}
\mbox{TIME(\texttt{edge\_pushing})} &\leq &\sum_{i=1}^\ell (d_i^* d^*+d^*+1) \nonumber\\
&=&O\left(d^* \sum_{i=1}^\ell d_i^*+ \ell\right). \label{eq:timecomplexity}
\end{eqnarray}

A consequence of this bound is that, if $f$ is linear, the complexity of \texttt{edge\_pushing} is that of the function
evaluation, a desirable property for Hessian algorithms.

\section{Computational experiments}\label{sec:computationalexperiments}
All tests were run on the 32-bit operating system Ubuntu 9.10, processor Intel 2.8~GHz, and 4~GB of RAM. All algorithms were coded
in C and C++. The algorithm \texttt{edge\_pushing} has been implemented as a driver of ADOL-C, and uses the taping and operator
overloading functions of ADOL-C \cite{ADOL-C}. The tests aim to establish a comparison between \texttt{edge\_pushing} and two
algorithms, available as drivers of ADOL-C v. 2.1, that constitute a well established reference in the field. These algorithms
incorporate the graph coloring routines of the software package \emph{ColPack}~\cite{Gebremedhin:inpreparation} and
the sparsity detection and Hessian-vector product procedures of ADOL-C~\cite{AndreaSparsity}. We shall denote them by the name of
the coloring scheme employed: Star and Acyclic.  Analytical properties of these algorithms, as well as numerical experiments with
them, have been reported in
\cite{AndreaSparsity,GebreHessianAuto}.

We have hand-picked fifteen functions from the CUTE collection \cite{Cute} and one --- \textsf{augmlagn} --- from \cite{Hock:1980} for the experiments. The selection was based on the following criteria: Hessian's sparsity pattern, scalability and sparsity. We wanted to cover a variety of patterns; to be able to freely change the scale of the function, so as to appraise the performance of the algorithms as the dimension grows; and we wanted to work with sparse matrices. The appendix of \cite{GowerMelloTechReportSept2010} presents results for dimension values $n$ in the set $5\,000$, $20\,000$, $50\,000$ and $100\,000$, but the tables in this section always refer to the $n=50\,000$ case, unless otherwise explicitly noted.

The list of functions is presented in Table~\ref{tab:testexamples}. The `Pattern' column indicates the type of sparsity pattern:
bandwidth\footnote{The bandwidth of matrix $M=(m_{ij})$ is the maximum value of $|i-j|$ such that
$m_{ij}\neq 0$.} (B $x$), arrow, box, or irregular pattern. The last two display the number of columns of the \emph{seed matrix}
produced by Star and Acyclic, for dimension equal to 50\kern2pt000. In order to report the performance of these algorithms, we
briefly recall their \emph{modus operandi}. Their first step, executed only once, computes a seed matrix $S$ via coloring methods,
such that the Hessian $f''$ may be recovered from the product $f''S$, which involves as many Hessian-vector products as the number
of columns of $S$. The latter coincides with the number of colors used in the coloring of a graph model of the Hessian. The
recovery of the Hessian boils down to the solution of a linear system. Thus the first computation of the Hessian takes necessarily
longer, because it comprises two steps, where the first one involves the coloring, and the second one deals with the calculation
of the actual numerical entries. In subsequent Hessian computations, only the second step is executed, resulting in a shorter run.
It should be noted that the number \revisedtext{of} colors is practically insensitive to changes in the dimension of the function
in the examples considered, with the exception of the functions with irregular patterns, \textsf{noncvxu2} and \textsf{ncvxbqp1}.

\begin{table}
\centering
\begin{tabular}{l|c|cc}	\hline								
&&\multicolumn{2}{c}{\#\ colors}\\
Name	&	Pattern	&	Star	&   Acyclic		\\	\hline \hline
\textsf{cosine} 	&	B 1	&	~3	&	2		\\	
\textsf{chainwoo}	&	B 2	&	~3	&	3		\\	
\textsf{bc4}		&	B 1		&	~3	&	2		\\	
\textsf{cragglevy}	&	B 1		&	~3	&	2		\\	
\textsf{pspdoc}	&	B 2	&	~5	&	3		\\	
\textsf{scon1dls}	&	B 2	&	~5	&	3		\\	
\textsf{morebv}	&	B 2	&	~5	&	3		\\	
\textsf{augmlagn}	&	$5\times5$ diagonal blocks		&	~5	&	5		\\	
\textsf{lminsurf}	&	B 5	&	11	&	6		\\	
\textsf{brybnd}	&	B 5	&	13	&	7		\\	
\textsf{arwhead}	&	arrow	&	~2	&	2		\\	
\textsf{nondquar}	& arrow + B 1 	&	~4	&	3		\\	
\textsf{sinquad}	&	frame + diagonal	&	~3	&	3		\\	
\textsf{bdqrtic}	& arrow + B 3 &	~8	&	5		\\	
\textsf{noncvxu2}	&	irregular	&	12	&	7		\\	
\textsf{ncvxbqp1}	&	irregular	&	12	&	7		\\	\hline
\end{tabular}									
\caption{Test functions}							
\label{tab:testexamples}														
\end{table}

Table~\ref{tab:starXacyclicXep} reports the times taken by \texttt{edge\_pushing} and by the first and second Hessian
computations by Star and Acyclic. It should be pointed out that Acyclic failed to recover the Hessian of \textsf{ncvxbqp1}, the
last function in the table. In the examples where \texttt{edge\_pushing} is faster than the second run of Star (resp., Acyclic),
we can immediately conclude that \texttt{edge\_pushing} is more efficient for that function, at that prescribed dimension. This
was the case in 14 (resp., 16) examples.  However, when the second run is faster than \texttt{edge\_pushing}, the corresponding
coloring method may eventually win, if the Hessians are computed a sufficient number of times, so as to compensate the initial
time investment. This of course depends on the context in which the Hessian is used, say in a nonlinear optimization code. Thus
the number of evaluations of Hessians is linked to the number of iterations of the code. The minimum time per example is
highlighted in Table~\ref{tab:starXacyclicXep}.

\begin{table}
\centering
\includegraphics[scale = 0.83]{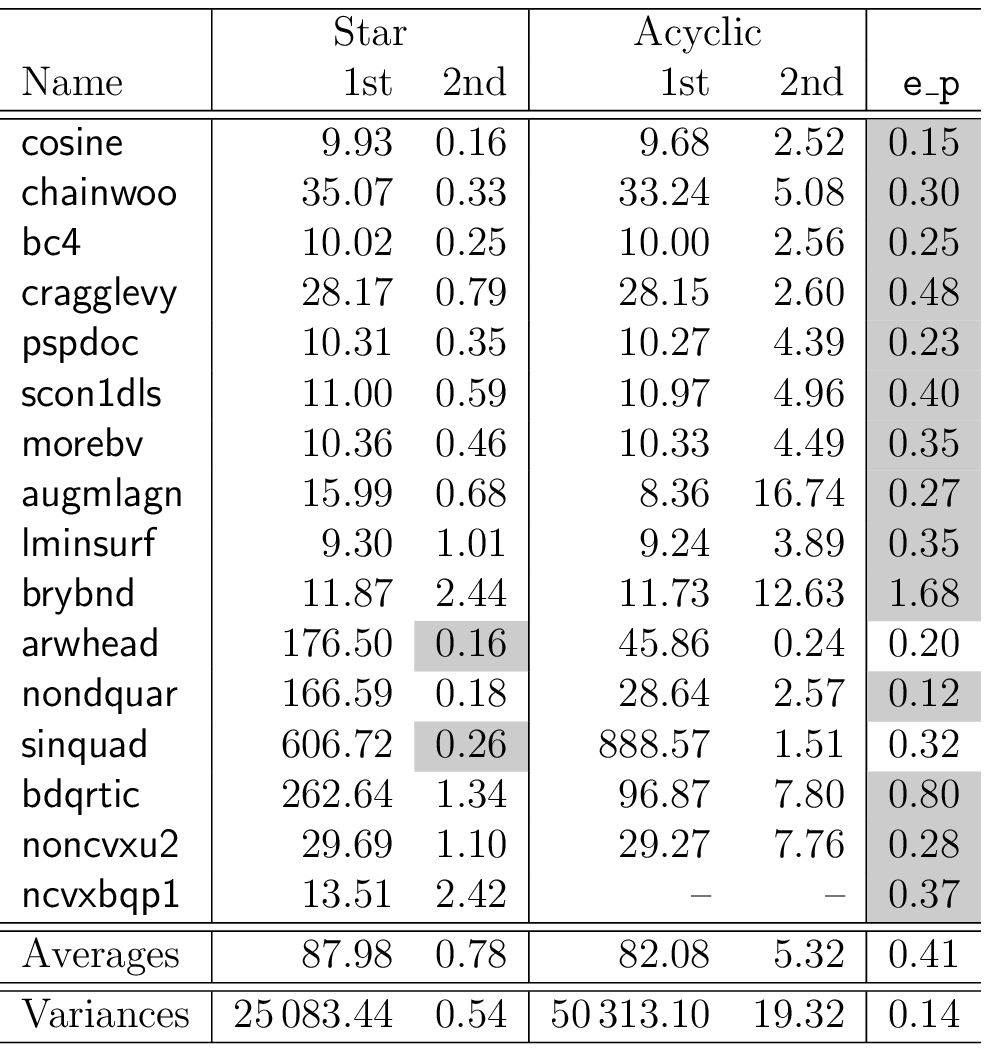}	
\caption{Runtimes in seconds for Star, Acyclic and \texttt{edge\_pushing}.}
\label{tab:starXacyclicXep}
\end{table}

Focusing on the two-stage Hessian methods, we see that Star always has fastest second runtimes. Only for function \textsf{sinquad}
is Star's first run faster than Acyclic's. Nevertheless, this higher investment in the first run is soon paid off, except for
functions \textsf{arwhead}, \textsf{nondquar} and \textsf{bdqrtic}, where it would require over 1600, 50 and 25, respectively,
computations of the Hessian to compensate the slower first run.  We can also see from Tables~\ref{tab:testexamples}
and~\ref{tab:starXacyclicXep} that Star's performance on the second run suffers the higher the number of colors needed to color
the Hessian's graph model, which is to be expected. Thus the second runs of \textsf{lminsurf}, \textsf{brybnd}, \textsf{bdqrtic},
\textsf{noncvxu2} and \textsf{ncvxbqp1} were the slowest of Star's. Notice that, although the Hessian of \textsf{bdqrtic} doesn't
require as many colors as the other four just mentioned, the function evaluation itself takes longer.

On a contrasting note, \texttt{edge\_pushing} execution is not tied to sparsity patterns and thus this algorithm proved to be more
robust, depending more on the density and number of nonlinear functions involved in the calculation. In fact, this is confirmed
by looking at the variance of the runtimes for the three algorithms, see the last row of Table ~\ref{tab:starXacyclicXep}. Notice
that \texttt{edge\_pushing} has the smallest variance. Furthermore, although Star was slightly faster than \texttt{edge\_pushing}
in the second run for the functions  \textsf{arwhead} and \textsf{sinquad}, the time spent in the first run was such that it would
require over 4\kern2pt000 and 10\kern2pt000, respectively, evaluations of the Hessian to compensate for the slower first run.

The bar chart in Figure~\ref{fig:barchartstarep}, built from the data in Table~\ref{tab:starXacyclicXep}, permits a graphical comparison of the performances of Star and \texttt{edge\_pushing}. Times for function \textsf{brybnd} deviate sharply from the remaining ones, it was a challenge for both methods. On the other hand, function \textsf{ncvxbq1} presented difficulties to Star, but not to \texttt{edge\_pushing}.

\begin{figure}
\centering
\includegraphics[width =\textwidth]{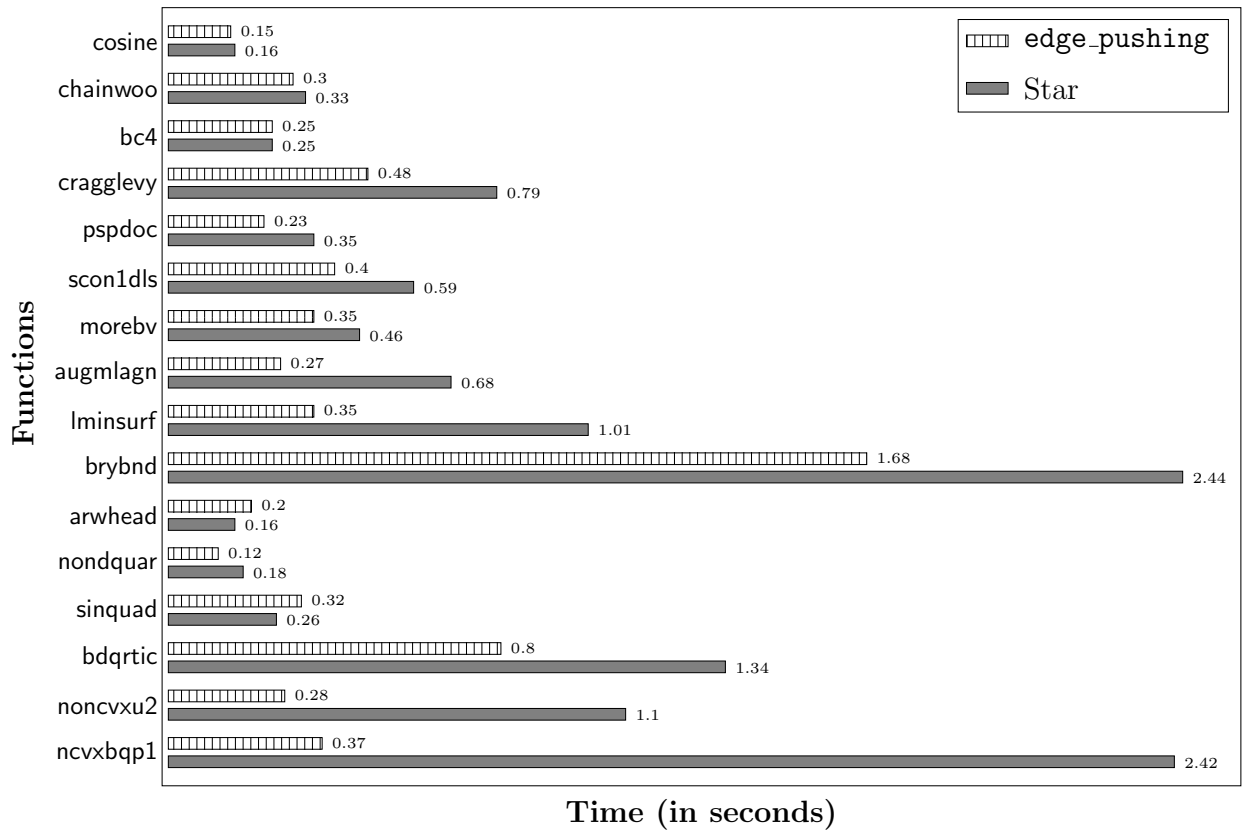}
\caption{Graphical comparison: Star versus \texttt{edge\_pushing}.}\label{fig:barchartstarep}
\end{figure}

The bar chart containing the runtimes of the three algorithms is made pointless by the range of runtimes of
Acyclic, much bigger than the other two. To circumvent this problem, we applied the base 10 log to the runtimes multiplied by 10
(just to make all logs positive). The resulting chart is depicted in Figure~\ref{fig:barchartacyclicstarep}.

\begin{figure}
\centering
\includegraphics[width =\textwidth]{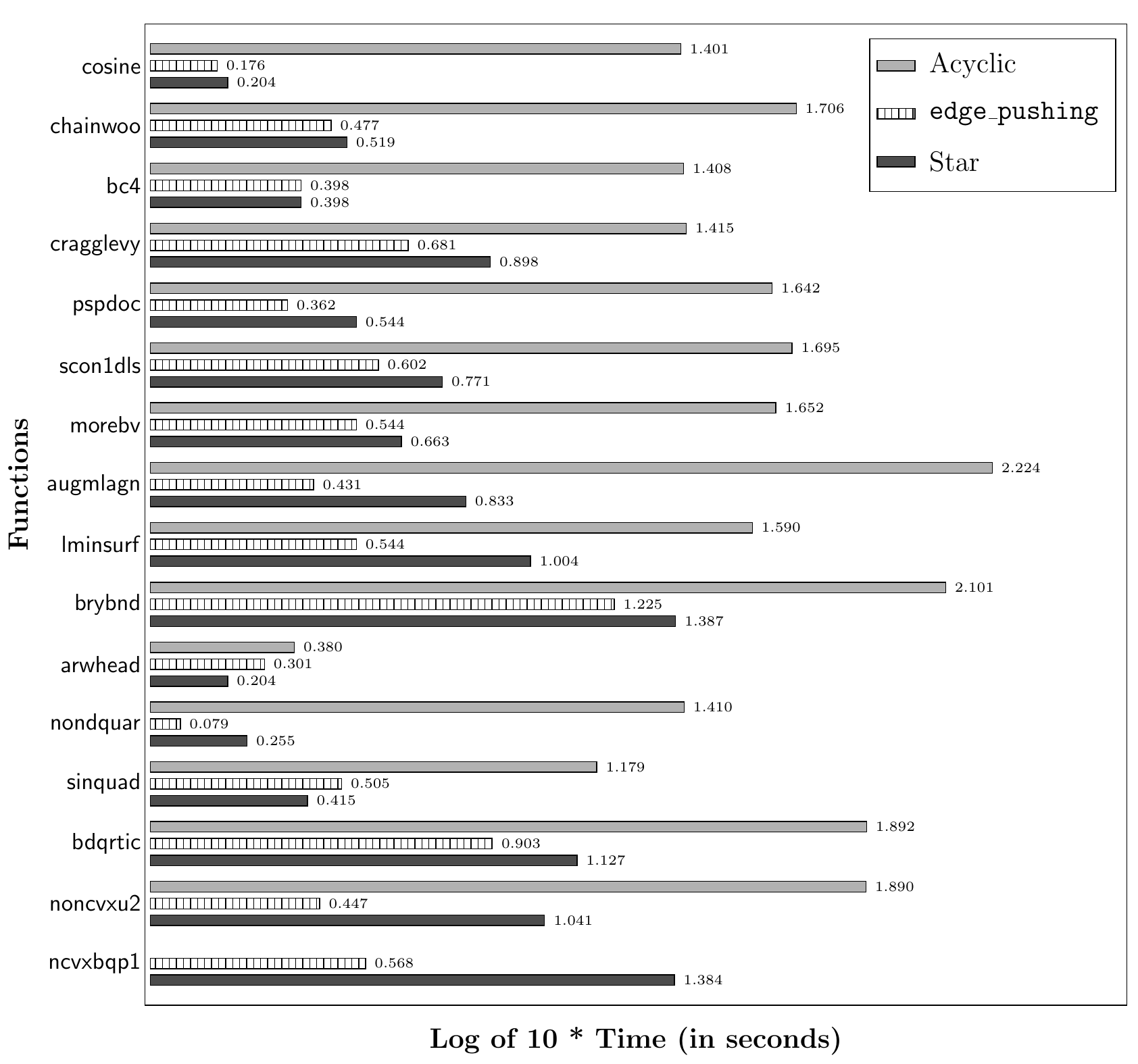}
\caption{Graphical comparison of times in log scale: Star, Acyclic and \texttt{edge\_pushing}.}\label{fig:barchartacyclicstarep}
\end{figure}

Although the results presented in Table~\ref{tab:starXacyclicXep} correspond to the dimension 50\kern2pt000 case, they
represented the general behavior of the algorithms in this set of functions. This is evidenced by the plots in
Figures~\ref{fig:sinecosdim} and~\ref{fig:brynondim}, that show the runtimes of \texttt{edge\_pushing} and Star on four functions
for dimensions varying from 5\kern2pt000 to 100\kern2pt000.

The functions \textsf{cosine}, \textsf{sinquad}, \textsf{brybnd} and \textsf{noncvxu2} were selected for these plots because they
exemplify the different phenomena we observed in the 50\kern2pt000 case. For instance, the performances of both
\texttt{edge\_pushing} and Star  are similar in the functions \textsf{cosine} and \textsf{sinequad}, and this has happened
consistently in all dimensions. Thus the dashed and solid lines in Figure~\ref{fig:sinecosdim} intertwine, and there is no
striking dominance of one algorithm over the other. Also, these functions presented no real challenges, and the runtimes in all
dimensions are low.

\begin{figure}
\centering
\includegraphics[width =11cm]{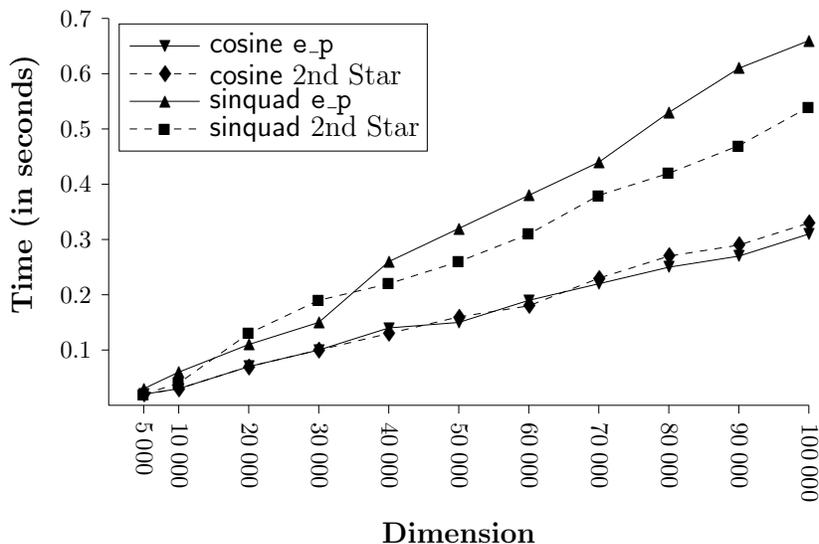}
\caption{Evolution of runtimes of \texttt{edge\_pushing} and Star (2nd run) with respect to dimension, for \textsf{cosine} and
\textsf{sinequad}.}\label{fig:sinecosdim}
\end{figure}

The function \textsf{brybnd} was chosen because it presented a challenge to all methods, and \textsf{ncvxu2} is the representative
of the functions with irregular Hessian sparsity patterns. The plots in Figure~\ref{fig:brynondim} show a consistent superiority
of \texttt{edge\_pushing} over Star for these two functions. All plots are close to linear, with the exception of the runtimes of
Star for the function \textsf{noncvxu2}. We observed that the  number of colors used to color the graph model of its Hessian
varied quite a bit, from 6 to 21. This highest number occurred precisely for the dimension 70\kern2pt000, the most
dissonant point in the series.

\begin{figure}
\centering
\includegraphics[width =11cm]{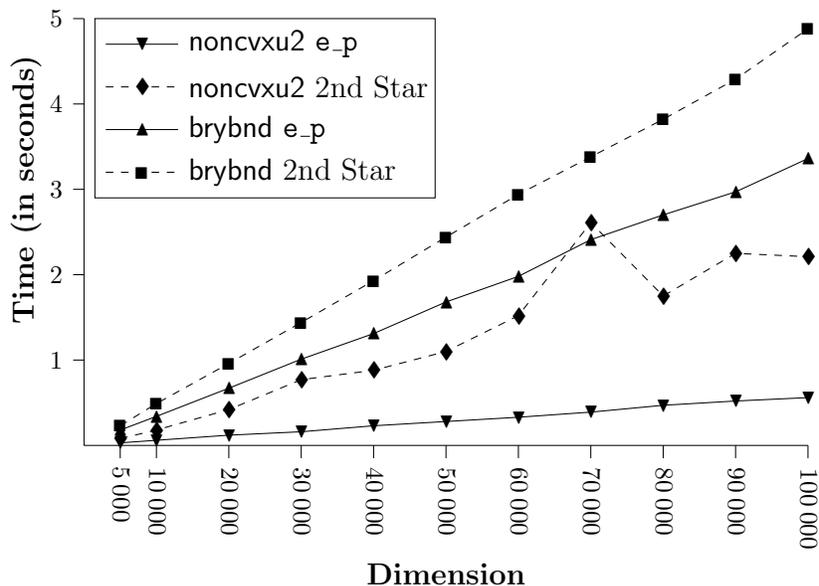}
\caption{Evolution of runtimes of \texttt{edge\_pushing} and Star (2nd run) with respect to dimension, for \textsf{noncvxu2} and
\textsf{brybnd}.}\label{fig:brynondim}
\end{figure}

The appendix of \cite{GowerMelloTechReportSept2010} contains the runtimes for the three methods, including first and second runs, for all functions, for dimensions 5\kern2pt000, 20\kern2pt000 and 100\kern2pt000.

\section{Conclusions and future research}\label{sec:conclusions}

The formula (\ref{eq:hessf}) for the Hessian obtained in Section~\ref{sec:HessianFormula} leads to new correctness proofs for existing Hessian computation algorithms and to the development of new ones. We also provided a graph model for the Hessian computation and both points of view inspired the construction of \texttt{edge\_pushing}, a new algorithm for Hessian computation that conforms to Griewank and Walther's Rule 16 of Automatic Differentiation \cite[p. 240]{Griewank:2008}:\\[8pt]
\centerline{\parbox{.67\textwidth}{\it%
The calculation of gradients by nonincremental reverse makes the corresponding computational graph 
symmetric, a property that should be exploited and maintained in accumulating Hessians. 
}}
\medskip

The new method is a truly reverse algorithm that exploits the symmetry and sparsity of the Hessian. It is a one-phase algorithm,
in the sense that there is no preparatory run where a sparsity pattern needs to be calculated that will be reused in all
subsequent iterations. This can be an advantage if the function involves many intermediate functions whose second derivatives are
zero in a sizable region, for instance $h(u)=(\max\{-u,0\})^2$. This type of function is used as a differentiable penalization of
the negative axis. It is not uncommon to observe the `thinning out' of Hessians over the course of nonlinear optimization, as the
iterations converge to an optimum, which obviously lies in the feasible region. If the sparsity structure is fixed at the
beginning, one cannot take advantaged of this slimming down of the Hessian.

 \texttt{edge\_pushing} was implemented as a driver of ADOL-C\cite{ADOL-C} and tested against two other algorithms, the Star and
Acyclic methods of ColPack~\cite{Gebremedhin:inpreparation}, also available as drivers of ADOL-C. Computational experiments were
run on sixteen functions of the CUTE collection~\cite{Cute}. The results show the strong promise of the new algorithm. When
compared to Star, there is a clear advantage of \texttt{edge\_pushing} in fourteen out of the sixteen functions. In the remaining
two the situation is unclear, since Star is a two-stage method and the first run can be very expensive. So even if its second run
is faster than \texttt{edge\_pushing}'s, one should take into account how many evaluations are needed in order to compensate the
first run. The answers regarding the functions \textsf{arwhead} and \textsf{sinquad} were over 4\kern2pt000 and 10\kern2pt000,
respectively, for dimension equal to 50\kern2pt000. These numbers grow with the dimension. Finally, it should be
noted that \textsf{edge\_pushing}'s performance was the more robust, and it wasn't affected by the lack of regularity in the
Hessian's pattern. 
 
We observed that Star was consistently better than Acyclic in all computational experiments. However, Gebremedhin et al. \cite{GebreHessianAuto} point out that Acyclic was better than Star in randomly generated Hessians and the real-world power transmission problem reported therein, while the opposite was true for large scale banded Hessians. It is therefore mandatory to test \texttt{edge\_pushing} not only on real-world functions, but also within the context of a real optimization problem. Only then can one get a true sense of the impact of using different algorithms for Hessian computation.

It should be pointed out that the structure of \texttt{edge\_pushing} naturally lends itself to parallelization, a task already underway. The opposite seems to be true for Star and Acyclic. The more efficient the first run is, the less colors, or columns of the seed matrix one has, and only the task of calculating the Hessian-vector products corresponding to $f''S$ can be seen to be easily parallelizable.

Another straightforward consequence of \texttt{edge\_pushing} is a sparsity pattern detection algorithm. This has already been implemented and tested, and will be the subject of another report.


\end{document}